\documentclass[12pt,a4paper]{amsart}
\usepackage{fullpage}
\usepackage{mathrsfs}
\usepackage{CJK}
\usepackage{geometry}
\usepackage{setspace}
\usepackage{amsfonts}
\usepackage{cite}
\usepackage{amssymb}
\usepackage{color, latexsym}
\usepackage{graphicx}
\usepackage{xy,caption2}
\usepackage{amsbsy,amscd}
\usepackage{fancyhdr}
\usepackage[colorlinks=true,
           citecolor=blue,
           linkcolor =red]{hyperref}
\xyoption{all}
\allowdisplaybreaks

\fontsize{9pt}{9pt}\selectfont

\def\bl{\boldsymbol{\ell}}
\def\bmu{\boldsymbol{\mu}}
\def\bgl{\boldsymbol{\lambda}}

\def\bc{\mathbf{c}}
\def\bi{\mathbf{i}}
\def\bk{\mathbf{k}}
\def\bq{\mathbf{Q}}

\def\bx{\mathbf{x}}
\def\by{\mathbf{y}}

\def\gl{\lambda}

\def\bz{\mathbb{Z}}

\def\ft{\mathfrak{t}}

\def\mpn{\mathscr{P}_{m,n}}

\newdimen\hoogte    \hoogte=16pt    
\newdimen\breedte   \breedte=16pt   
\newdimen\dikte     \dikte=1.0pt    
\newenvironment{point}[2]%
  {\vspace{0.5\jot}\ifx*#2\let\pointlabel\relax\else\def\pointlabel{#2}\fi
   \refstepcounter{equation}\trivlist
   \item[\hskip\labelsep\theequation.
         \ifx\pointlabel\relax\else\space\pointlabel\space\fi]
   \ignorespaces #1
  \vspace{0.5\jot}}{\relax}
\newenvironment{Young}{\begingroup
       \def\vr{\vrule height0.6\hoogte width\dikte depth 0.2\hoogte}
       \def\fbox##1{\vbox{\offinterlineskip
                    \hrule height0.6\dikte
                    \hbox to \breedte{\vr\hfill$##1$\hfill\vr}
                    \hrule height0.6\dikte}}
       \vtop\bgroup \offinterlineskip \tabskip=-\dikte \lineskip=-\dikte
            \halign\bgroup &\fbox{##\unskip}\unskip  \crcr}
     {\egroup\egroup\endgroup}

\def\diagram(#1){\begin{Young}#1\cr\end{Young}}

\def\Tritab(#1|#2|#3){\begin{Young}#1\cr\end{Young}\,;\,\
                              \begin{Young}#2\cr\end{Young}\,;\,\
                              \begin{Young}#3\cr\end{Young}}
\def\Twotab(#1|#2){\begin{Young}#1\cr\end{Young}\,;\,
\begin{Young}#2\cr\end{Young}}

\newenvironment{colorYoung}{\begingroup
       \def\vr{\vrule height0.8\hoogte width\dikte depth 0.2\hoogte}
       \def\fbox##1{\vbox{\offinterlineskip
                    \hrule height1.0\dikte
                    \hbox to 1.08\breedte{\vr\hfill$##1$\hfill\vr}
                    \hrule height1.0\dikte}}
       \vtop\bgroup \offinterlineskip \tabskip=-\dikte \lineskip=-\dikte
            \halign\bgroup &\fbox{##\unskip}\unskip  \crcr}
     {\egroup\egroup\endgroup}

\def\Tricolor(#1|#2|#3){\begin{colorYoung}#1\cr\end{colorYoung}\,;\,\
                              \begin{colorYoung}#2\cr\end{colorYoung}\,;\,\
                              \begin{colorYoung}#3\cr\end{colorYoung}}
 \def\Twocolor(#1|#2){\begin{colorYoung}#1\cr\end{colorYoung}\,;\,\
                              \begin{colorYoung}#2\cr\end{colorYoung}}
 \def\colordiagram(#1){\begin{colorYoung}#1\cr\end{colorYoung}}

 \newenvironment{wtYoung}{\begingroup
       \def\vr{\vrule height0.8\hoogte width\dikte depth 0.2\hoogte}
       \def\fbox##1{\vbox{\offinterlineskip
                    \hrule height1.0\dikte
                    \hbox to 2.1\breedte{\vr\hfill$##1$\hfill\vr}
                    \hrule height1.4\dikte}}
       \vtop\bgroup \offinterlineskip \tabskip=-\dikte \lineskip=-\dikte
            \halign\bgroup &\fbox{##\unskip}\unskip  \crcr}
     {\egroup\egroup\endgroup}
\def\Triwt(#1|#2|#3){\begin{wtYoung}#1\cr\end{wtYoung}\,;\,\
                              \begin{wtYoung}#2\cr\end{wtYoung}\,;\,\
                              \begin{wtYoung}#3\cr\end{wtYoung}}

\makeatother
\numberwithin{equation}{section}
\swapnumbers
\newtheorem{theorem}[equation]{Theorem}
\newtheorem{lemma}[equation]{Lemma}
\newtheorem{proposition}[equation]{Proposition}
\newtheorem{corollary}[equation]{Corollary}

\theoremstyle{definition}
\newtheorem{definition}[equation]{Definition}
\newtheorem{example}[equation]{Example}
\theoremstyle{remark}
\newtheorem{remark}[equation]{Remark}
\begin{document}
\setlength{\itemsep}{1\jot}
\fontsize{13}{\baselineskip}\selectfont
\setlength{\parskip}{0.3\baselineskip}
\vspace*{0mm}
\title[RSK superinsertion and super Frobenius formulae]{\fontsize{9}{\baselineskip}\selectfont RSK superinsertion and super Frobenius formulae}
\author{Deke Zhao}
\address{\bigskip\hfil\begin{tabular}{l@{}}
          School of Applied Mathematics\\
            Beijing Normal University at Zhuhai \\
            Zhuhai, 519087 China\\
             E-mail: \it deke@amss.ac.cn \hfill
          \end{tabular}}
\subjclass[2010]{Primary 05E10, 05E05; Secondary 20C99, 20C15}
\keywords{Frobenius formula; Complex reflection group; Cyclotomic Hecke algebra; Robinson-Schensted-Knuth (super)insertion; Supersymmetric function.}
\vspace*{-3mm}
\begin{abstract}In this paper we extend the Robinson-Schensted-Knuth (RSK) superinsertion algorithm to hook-multipartitions and derive the super Frobenius formula for the characters of cyclotomic Hecke algebras  (Linear and Multilinear Algebra, DOI: 10.1080/03081087.2019.1663140) via the (RSK) superinsertion algorithm. In particular, we obtain a new proof of Mitsuhashi's super Frobenius formula for the characters of Iwahori-Hecke algebras.
\end{abstract}
\maketitle
\vspace*{-5mm}
\section{Introduction}
In \cite{Frobenius}, Frobenius gave a formula of computing the characters of symmetric groups, which is often referred as the Frobenius formula. In his study of representations of the general linear group, Schur \cite{Schur-d,Schur} showed the Frobenius formula can be obtained by the classical Schur-Weyl reciprocity. Inspired by Schur's classical work, Ram \cite{Ram} got the Frobenius formula for the characters of Iwahori-Hecke algebras of type $A$ by applying the Schur-Weyl reciprocity established by Jimbo \cite{Jimbo}; Shoji \cite{S} obtained the Frobenius formula for the characters of cyclotomic-Hecke algebras by making use the  Schur-Weyl reciprocity  given by Sakamoto-Shoji \cite{SS}; Mitsuhashi \cite{M2010} showed the super Frobenius formula for the characters of  Iwahori-Hecke algebras of type $A$  by virtue of the Schur-Weyl reciprocity between the Iwahori-Hecke algebras of type $A$ and the quantum superalgebra \cite{Moon,Mit};   Recently the author \cite{Z-Frob} obtains the super Frobenius formulas for the characters of cyclotomic Hecke algebras by employing the super Schur-Weyl reciprocity between the cyclotomic Hecke algebra and the quantum superalgebra proved in \cite{Zhao18}.

In \cite{Ram98}, Ram observed that the Roichman formula \cite{Roichman} and his Frobenius formula can be acquired by applying the Robinson-Schensted-Knuth (RSK) insertion algorithm. Following Ram's \textit{loc. cit.} work,  Cantrell et. al. \cite{CHM} provided a new proof of Shoji's  Frobenius formula via the modified RSK insertion algorithm for multipartitions. Note that there are super-analogues of RSK insertion algorithm for specific partitions (see e.g. \cite[\S2.5]{B-Regev}, \cite[\S4]{DR}). It is natural to ask whether the super Frobenius formulae in \cite{M2010,Z-Frob} can be showed by adapting the arguments of \cite{Ram98,CHM}.

The purpose of the paper is to derive the super Frobenius formulae in terms of weights on standard tableaux by applying the modified RSK superinsertion algorithm for hook (multi)partitions. A point should be note that our strategy is very similar to the one in \cite{Ram98,CHM}, that is, let $\mpn$ be the set of multipartitions of $n$, for $\bmu\in\mpn$, we first define the $\bmu$-weight of parity-integer sequences. Therefore we can rewrite the super Frobenius formula as a sum over $\bmu$-weighted parity-integer sequences (see Corollary~\ref{Cor:q-mu-wt}).
 Secondly, we introduce the RSK superinsertion algorithm for hook multipartitions, which gives a bijection between pairs specific tableaux and parity-integer sequences (see \S\ref{Subsec:RSK}). We then define the $\bmu$-weight $\mathrm{wt}_{\bmu}$ of standard tableaux, which is exactly the $\bmu$-weight of parity-integer sequences up to the RSK superinsertion algorithm. Thus we arrive at the super Frobenius formula for the characters of cyclotomic Hecke algebras (see Theorem~\ref{Them:q=wt-Schur})\begin{equation*}
   q_{\bmu}(\bx/\by;q,\bq)=\sum_{\bgl\in\mpn}\biggl(\sum_{T\in\mathrm{std}
   (\bgl)}\mathrm{wt}_{\bmu}(T)\biggr)S_{\bgl}(\bx/\by),
 \end{equation*}
where $\mathrm{std}(\bgl)$ is the set of standard $\bgl$-tableaux and $S_{\bgl}(\bx/\by)$ is the supersymmetric Schur function associated to  $\bgl$. As a direct application, we derive a formula to computes the irreducible characters of cyclotomic Hecke algebras as a sum of $\bmu$-weights over standard tableaux (see Corollary~\ref{Cor:Character}).

Let us remark that if $m=1$ then the super Frobenius formula for the characters of cyclotomic Hecke algebras is  Mitsuhashi's super Frobenius formula for the characters of Iwahori-Hecke algebras, therefor we obtain a new proof of Mitsuhashi's super Frobenius formula (see Remark~\ref{Remark:m=1}), which is super-extension of Ram's work \cite{Ram98}. On the other hand, if we let $q=1$ and $\mathbf{Q}=(\varsigma, \varsigma^2,\ldots, \varsigma^m)$, where $\varsigma$ is a fixed primitive $m$-th root of unity, then the super Frobenius formula gives the super Frobenius formula for the complex reflection group of type $G(m,1,n)$ (see \cite[Remark~4.12]{Z-Frob}). So we derive a formula to computes the irreducible characters of  complex reflection group of type $G(m,1,n)$ as a sum of $\bmu$-weights over standard tableaux (see Remark~\ref{Remark:q=1}).

The layout of the paper is as follows. In Section~\ref{Sec:Cyc-Hecke-algebras} we review briefly the cyclotomic Hecke algebras and fix our notations on (multi)partitions.  Section~\ref{Sec:super-Frobenius} devotes to introduce the supersymmetric Schur functions and power sum supersymmetric functions associated to multipartitions and the super Hall-Littlewood functions. In particular, we present the super Frobenius formula for the characters of cyclotomic Hecke algebras. We define the ($\bmu$-)weighted parity-integer sequences  and give a combinatorial version of the Frobenius formula in Section~\ref{Sec:mu-weight-i}. The last section deal with the RSK superinsertion for hook-(multi)partitions and present the new proof of the super Frobenius formula via the RSK superinsertion  algorithm for multipartitions.

Throughout the paper, $\mathbb{K}=\mathbb{C}(q,\bq)$ is the field of rational function in indeterminates $q$ and  $\bq=(Q_1, \ldots, Q_m)$; $k_1, \ldots, k_m$, $\ell_1, \ldots, \ell_m$ are non-negative integers with  $\sum_{i=1}^mk_i=k$, $\sum_{i=1}^m\ell_i=\ell$. We denote by  $\bk=(k_1, \ldots, k_m)$, $\bl=(\ell_1, \ldots, \ell_m)$ and let $d_i=d_{i-1}+k_i+\ell_i$ ($d_0=0$) for $i=1, \ldots, m$. For $i=1, \ldots, k+\ell$, we define the \textit{parity function}  $i\mapsto \overline{i}$ by setting \begin{equation*} \overline{i}=\left\{\begin{array}{ll}\overline{0}, &\hbox{ if }d_{a-1}< i\leq d_{a-1}+k_{a}\text{ for some }1\leq a\leq m;\\\overline{1}, & \hbox{ if }d_{a}-\ell_{a}<i\leq d_{a}\text{ for some }1\leq a\leq m;.\end{array}\right.\end{equation*} The \textit{color} of $i$ is defined to be $c(i)=a$ if $d_{a-1}<i\leq d_a$ for some $1\leq a\leq m$.
\subsection*{Acknowledgements}Part of this work was carried out while the author was visiting the Chern Institute of Mathematics at the Nankai University. The authors would like to thank Professor Chengming Bai for his hospitality during his visit. 
The author was supported partly by the National Natural Science Foundation of China
(Grant No. 11571341, 11671234, 11871107).
\section{Cyclotomic Hecke algebras}\label{Sec:Cyc-Hecke-algebras}
In this section we review briefly some facts about cyclotomic Hecke algebras for latter using and fix our notations on (multi)partitions.
\begin{point}{}*\label{subsec:G(m,1,n)}
 Let $W_{m,n}$ be the complex reflection group of type $G(m,1,n)$. According to \cite{shephard-toda}, $W_{m,n}$ is generated by $s_0, s_1, \dots, s_{n-1}$ with relations
  \begin{align*}&s_0^m=1,\quad s_1^2=\cdots=s_{n-1}^2=1,&&\\
  &s_0 s_1s_0 s_1=s_1s_0 s_1s_0,&&\\
 & s_is_j=s_js_i,&&\text{ if } |i-j|>1,\\
&s_is_{i+1}s_i=s_{i+1}s_{i}s_{i+1},&& \text{ for }1\leq i<n-1.\end{align*}
 It is well-known that $W_{m,n}\cong(\mathbb{Z}/m\bz)^{n}\rtimes \mathfrak{S}_{n}$, where $s_1, \dots, s_{n-1}$ are generators of the symmetric group $\mathfrak{S}_{n}$ of degree $n$ corresponding to transpositions $(1\,2)$, $\ldots$, $(n\!-\!1\,n)$. For $a=1,\ldots, n$,  let $t_a=s_{a-1}\cdots s_1s_0s_1\cdots s_{a-1}$. Then $t_1, \cdots, t_n$ are generators of $(\mathbb{Z}/m\bz)^{n}$ and any element $w\in W_{m,n}$ can be written in a unique way as $w=t_1^{c_1}\cdots t_n^{c_n}\sigma$, where $\sigma\in \mathfrak{S}_{n}$ and $c_i$ are integers such that $0\leq c_i<m$.
 \end{point}

\begin{point}{}*\label{Subsec:Cyc-Hecke-algebras}The \textit{cyclotomic Hecke algebras} or \textit{Ariki-Koike} algebras were introduced independently by Cherednik \cite{Cherednik}, Brou\'{e}-Malle \cite{Broue-Malle}, and Ariki-Koike \cite{AK}, it is the unital
associative $\mathbb{K}$-algebra $H=H_{m,n}(q,\mathbf{Q})$ generated by
$T_0,T_1,\dots,T_{n}$ and subject to relations
\begin{align*}&(T_0-Q_1)\dots(T_0-Q_m)=0,&&\\
&T_0T_1T_0T_1=T_1T_0T_1T_0,&&\\
&T_i^2=(q-q^{-1})T_i+1, &&\text{ for }1\leq i< n,\\
&T_iT_j=T_jT_i, &&\text{ for }|i-j|>2,\\
 &T_iT_{i+1}T_i=T_{i+1}T_{i}T_{i+1}, &&\text{ for }1\leq i\leq n-2.\end{align*}
If $s_{i_1}s_{i_2}\cdots s_{i_k}$ is a reduced expression for $\sigma\in \mathfrak{S}_n$. Then $T_{\sigma}:=T_{i_1}T_{i_2}\cdots T_{i_k}$ is independent of the choice of reduced expression and  $\{T_{\sigma}|\sigma\in \mathfrak{S}_n\}$ is a linear basis of the subalgebra $H_n(q)$ of $H$ generated by $T_1, \ldots, T_{n-1}$, which is the generic Iwahori-Hecke algebra of type $A$.

Let $\phi: \mathbb{K}\rightarrow \mathbb{C}$ be the specialization homomorphism defined by $\phi(q)=1$ and $\phi(Q_i)=\varsigma^i$ for each $i$, where $\varsigma$ is a fixed primitive $m$-th root of unity. By the specialization homomorphism $\phi$, one obtains $\mathbb{C}\otimes_{\mathbb{K}}H\cong \mathbb{C}W_{m,n}$, that is, $H$ can be thought as a deformation of the group algebra $\mathbb{C}W_{m,n}$.
\end{point}

\begin{point}{}*We will need the following presentation of $H$ due to Shoji \cite[Theorem~3.7]{S}.
Let $\Delta$ be the determinant of the Vandermonde matrix $V(\mathbf{Q})$ of degree $m$ with $(a,b)$-entry $Q_b^a$ for $1\leq b\leq m$, $0\leq a<m$. Clearly, we can write $V(\mathbf{Q})^{-1}=\Delta^{-1}V^*(\mathbf{Q})$, where $V^*(\mathbf{Q})=(v_{ba}(\mathbf{Q}))$ and $v_{ba}(\mathbf{Q})$ is a polynomial in $\mathbb{Z}[\mathbf{Q}]$. For $1\leq c\leq m$, we define a polynomial $F_c(X)$ with a variable $X$ with coefficients in $\mathbb{Z}[\mathbf{Q}]$ by
\begin{equation*}
  F_c(X)=\sum_{0\leq i<m}v_{ci}(\mathbf{Q})X^i.
\end{equation*}
Then $H$ is (isomorphic to) the associative $\mathbb{K}$-algebra generated by $T_1, \ldots, T_{n-1}$ and $\xi_1, \ldots, \xi_n$ subject to
\begin{align*}& (T_i-q)(T_i+q^{-1})=0,  && 1\leq i< n,\\
          &(\xi_i-Q_1)\cdots(\xi_i-Q_m)=0, && 1\leq i\leq n, \\
           &T_iT_{i+1}T_{i}=T_{i+1}T_iT_{i+1},&&  1\leq i< n-1,\\
            &T_iT_j=T_jT_i,&& |i-j|\geq 2, \\
         & \xi_i\xi_j=\xi_j\xi_i,&& 1\leq i,j\leq n, \\
         &T_j\xi_i=\xi_iT_j,&&  i\neq j-1, j,\\
 &T_j\xi_{j}=\xi_{j-1}T_j+\Delta^{-2}\sum_{a<b}(Q_{a}-Q_{b})(q-q^{-1})F_{a}(\xi_{j-1})F_{b}(\xi_j),&\\
 &T_j\xi_{j-1}=\xi_{j}T_j-\Delta^{-2}\sum_{a<b}(Q_{a}-Q_{b})(q-q^{-1})F_{a}(\xi_{j-1})F_{b}(\xi_j).&
\end{align*}
\end{point}

\begin{point}{}* Recall that a composition (resp. partition) $\gl=(\gl_1, \gl_2, \ldots)$ of $n$, denote $\gl\models n$ (resp. $\gl\vdash n$) is a sequence (resp. weakly decreasing sequence) of  nonnegative integers such that $|\gl|=\sum_{i\geq1}\gl_i=n$ and we write $\ell(\gl)$ the {\it length} of $\gl$, i.e. the number of nonzero parts of $\gl$. A {\it multicomposition} (resp. {\it multipartition}) of $n$ is an ordered tuple $\bgl=(\gl^{(1)}; \ldots; \gl^{(m)})$ of compositions (resp. partitions) $\lambda^{(i)}$ such that $|\bgl|=\sum_{i=1}^m|\lambda^{(i)}|=n$. We refer to $\gl^{(c)}$ as the $c$-component of $\bgl$ and denote by $\mpn$ the set of all multipartitions of $n$.

Recall that the \textit{Young diagram} $Y(\bgl)$ of multipartition $\bgl$ is the set
 \begin{equation*}
   Y(\bgl)=\{(i,j,c)\in\bz_{>0}\times\bz_{>0}\times \mathbf{m}|1\le j\le\lambda^c_i\},
 \end{equation*}
where $\mathbf{m}=\{1, \ldots, m\}$. We will identify a (multi)partition with its Young diagram. For example, the Young diagram of $\bgl=\left( (2,1,1);(3,2,2,1);(4,3,1)\right)\in\mathscr{P}_{3,20}$ is
\begin{equation*}
  \bgl=\biggl(\,\Tricolor(&\cr\cr|&&\cr&\cr&\cr|&&&\cr&&\cr)\,\biggr).
\end{equation*}
\end{point}

\begin{point}{}*
Recall that a partition $\gl=(\gl_1, \gl_2, \cdots)\vdash n$ is said to be a \textit{$(k, \ell)$-hook partition} of $n$ if $\gl_{k+1}\leq \ell$. We let  $H(k,\ell;n)$ denote the set of all $(k,\ell)$-hook partitions of $n$, that is
\begin{eqnarray*}\label{Def:hook}
H(k,\ell;n)=\{\gl=(\gl_1,\gl_2,\cdots)\vdash n\mid \gl_{k+1}\le \ell\}.
\end{eqnarray*}
Further, a multipartition $\bgl=(\gl^{(1)}; \ldots; \gl^{(m)})$ of $n$ is said to be a \textit{$(\bk,\bl)$-hook multipartition} of $n$ if $\gl^{(c)}$ is a $(k_i,\ell_i)$-hook partition for $c=1, \ldots,m$. We denote by $H(\bk|\bl; m,n)$ the set of all  $(\bk,\bl)$-hook multipartition of $n$. For example, $\bgl=\left( (2,1,1);(3,2,2,1);(4,3,1)\right)$ is a $(1|1,1|2,1|3)$-hook $3$-multipartition of $20$.
\end{point}

\begin{point}{}*\label{subsec:standard-elements}Given a composition $\mathbf{c}=(c_1, \ldots, c_b)$ of $ n$, we denote by  $\mathfrak{S}_n^{(i)}$  the subgroup of $\mathfrak{S}_n$ generated by $s_j$ such that $a_{i-1}+1<j\leq a_i$, where $a_i=\sum_{j=1}^ic_j$ for $1\leq i\leq b$.  Then $\mathfrak{S}_n^{(i)}\cong \mathfrak{S}_{c_i}$. Now we define a parabolic subgroup $\mathfrak{S}_{\mathbf{c}}$ of $\mathfrak{S}_n$ by
\begin{equation*}
 \mathfrak{S}_{\mathbf{c}}=\mathfrak{S}_n^{(1)}\times \mathfrak{S}_n^{(2)}\times\cdots\times \mathfrak{S}_n^{(b)}.
\end{equation*}In other words,  $\mathfrak{S}_{\mathbf{c}}$ is the Young subgroup of $\mathfrak{S}_n$ associated to $\mathbf{c}$.

Let $W_{\mathbf{c}}$ be the subgroup of $W_{m,n}$ obtained as the semidirect product of $\mathfrak{S}_{\mathbf{c}}$ with $(\mathbb{Z}/m\bz)^{n}$. Then $W_{\mathbf{c}}$ can be written as
\begin{equation*}
  W_{\mathbf{c}}=W^{(1)}\times W^{(2)}\times\cdots\times W^{(b)},
\end{equation*}
where $W^{(1)}$ is the subgroup of $W_{m,n}$ generated by $\mathfrak{S}_n^{(i)}$ and $t_j$ such that $a_{i-1}<j\leq a_i$. Then $W^{(i)}\cong W_{m,c_i}$, which enables us to yield a natural embedding
\begin{equation*}
  \theta_{\mathbf{c}}:  W_{m,c_1}\times\cdots\times W_{m,c_b}\hookrightarrow  W_{m,n}
\end{equation*}
for each composition $\mathbf{c}$ of $n$.

For positive integer $i$, we define
\begin{eqnarray*}
  w(1,i)=t_1^i, & & w(a,i)=t_a^is_{a-1}\cdots s_1, \quad 2\leq a\leq n.
\end{eqnarray*}
Let $\mu=(\mu_1,\mu_2,\cdots)$ be a partition of n. We define
\begin{eqnarray*}
  w(\mu,i)&=&w(\mu_1,i)\times  w(\mu_2,i)\times\cdots
\end{eqnarray*}
with respect to the embedding $\theta_{\mu}$.  More generally, for $\bmu=(\mu^{(1)}, \ldots, \mu^{(m)})\in\mpn$,  we define
\begin{eqnarray*}
  w(\bmu) &=& w(\mu^{(1)},1)w(\mu^{(2)},2)\cdots w(\mu^{(m)},m).
\end{eqnarray*}
Then $\{w(\bmu)|\bmu\in\mpn\}$ is a set of conjugate classes representatives for $W_{m,n}$ (see \cite[4.2.8]{JK}).
\end{point}

\begin{point}{}*\label{subsec:standard-elements}Now we define the \textit{standard element} $T(\bmu)\in H$ of type $\bmu\in \mathscr{P}_{m,n}$ as follows.
First for $a\geq 1$, we put $T(a,i)=\xi_a^iT_{a-1}\cdots T_1$.  Then for a partition $\mu=(\mu_1,\mu_2,\ldots)\vdash n$, we define
\begin{equation*}
  T(\mu,i)=T(\mu_1,i)\times T(\mu_2,i)\times\cdots
\end{equation*}
Finally, for $\bmu=(\mu^{(1)};\mu^{(2)}; \ldots; \mu^{(m)})\in\mpn$, we define $T(\bmu)\in\mathcal{H}$ by
\begin{eqnarray}\label{Equ:T-bmu}
   &&T(\bmu)=T(\mu^{(1)},1)\times T(\mu^{(2)},2)\times\cdots\times T(\mu^{(m)},m).
\end{eqnarray}
More generally, we define $T(w)=\xi_1^{c_1}\cdots\xi_n^{c_n}T_\sigma$ for $w=t_1^{_1}\cdots t_n^{c_n}\sigma\in W_{m,n}$.
\end{point}

It is well-known that the irreducible representations of $\mathbb{C}W_{m,n}$ and $H$ are indexed by $\mpn$. We denote by $S_1^{\bgl}$ (resp. $S_q^{\bgl}$) the irreducible representations of $\mathbb{C}W_{m,n}$ (resp. $H$) corresponding to $\bgl\in\mpn$ and by $\chi_1^{\bgl}$ (resp. $\chi_q^{\bgl}$) its irreducible character. Furthermore, it is known that the characters $\chi_1^{\bgl}$ (resp. $\chi_q^{\bgl}$) are completely by their values on $w(\bmu)$ (resp. $T(\bmu)$) for all $\bmu\in\mpn$. For simplicity, we write $\chi_1^{\bgl}(\bmu)=\chi_1^{\bgl}(w(\bmu))$ and $\chi_q^{\bgl}(\bmu)=\chi_q^{\bgl}(T(\bmu))$ repectively.

\section{The super Frobenius formula}\label{Sec:super-Frobenius}

In this section we first recall the definitions of the supersymmetric Schur functions and power sum supersymmetric functions indexed by multipartitions introduced in \cite[\S3]{Z-Frob}, then we  present the super Frobenius formula for the characters of cyclotomic Hecke algebras. Here we will follow \cite{Macdonald} with respect to our notation about symmetric functions unless otherwise stated.

\begin{point}{}*\label{Point:index-bijection}
 Let $\mathbf{x}, \mathbf{y}$ be sets of $k, \ell$ indeterminates respectively as following:
\begin{eqnarray*}
  \mathbf{x}^{(i)}&=&\left\{x^{(i)}_1, \ldots, x_{k_i}^{(i)}\right\},\quad 1\leq i\leq m,\\
  \mathbf{y}^{(i)}&=&\left\{y^{(i)}_1, \ldots, y_{\ell_i}^{(i)}\right\},\quad 1\leq i\leq m,\\
    \mathbf{x}&=&\mathbf{x}^{(1)}\cup\cdots\cup \mathbf{x}^{(m)},\\
    \mathbf{y}&=&\mathbf{y}^{(1)}\cup\cdots\cup \mathbf{y}^{(m)}.
\end{eqnarray*}
We say that the variables in $\bx^{(i)}\cup\by^{(i)}$ are of color $i$ and set $\mathrm{deg}\,x=\bar{0}$, $\mathrm{deg}\,y=\bar{1}$ for $x\in\bx,y\in \by$.
We identify the variables  $x_1^{(1)}$, $\ldots$, $y^{(m)}_{\ell_m}$ with the variables $x_1$, $\ldots$, $x_{k}, y_1, \ldots, y_{\ell}$, and the variables $z_1, \ldots, z_{k+\ell}$ as follows:
\begin{equation*}\label{Equ:Basis-index}
  \begin{array}{cccccccccccccc}
  x_{1}^{(1)}& \cdots & x_{k_1}^{(1)} &y_{1}^{(1)} & \cdots&y_{\ell_1}^{(1)}&\cdots&\cdots&
  x_{1}^{(m)}& \cdots & x_{k_m}^{(m)} &y_{1}^{(m)}&  \cdots&y_{\ell_m}^{(m)}\\
  \updownarrow& \vdots & \updownarrow & \updownarrow  & \vdots&\updownarrow &\cdots&\cdots& \updownarrow& \vdots & \updownarrow & \updownarrow  & \vdots&\updownarrow \\
  x_1& \cdots & x_{k_1} & y_{1} &  \cdots&y_{\ell_1} &\cdots&\cdots&
    x_{k\!-\!k_m\!+\!1}& \cdots & x_{k} & y_{\ell\!-\!\ell_m\!+\!1} &  \cdots&y_{\ell}\\
  \updownarrow& \vdots & \updownarrow & \updownarrow  & \vdots&\updownarrow &\cdots&\cdots& \updownarrow& \vdots & \updownarrow & \updownarrow  & \vdots&\updownarrow \\
  z_1& \cdots & z_{k_1} & z_{k_1+1} &  \cdots&z_{d_1} &\cdots&\cdots&
    z_{d_{m-1}\!+\!1}& \cdots & z_{d_m-\ell_m} & z_{d_m\!-\!\ell_m\!+\!1} &  \cdots&z_{k+\ell}.
\end{array}
\end{equation*}
 We linearly order the variables $x_1^{(1)}, \ldots, x^{(m)}_{k_m}$, $y_1^{(1)}, \ldots, y^{(m)}_{\ell_m}$ by the rule  $x_{*}^{(i)}<y_{*}^{(i)}<x_{*}^{(i+1)}$ and \begin{eqnarray*}
  x_{a}^{(i)}<x_{b}^{(j)}&\text{ if and only if }& i<j\text{ or }i=j\text{ and }a<b;\\
   y_{a}^{(i)}<y_{b}^{(j)}&\text{ if and only if }& i<j\text{ or }i=j\text{ and }a<b,                                                         \end{eqnarray*}
  or equivalently $z_1<z_2<\ldots<z_{k+\ell}$.
\end{point}

\begin{point}{}*
Let $\Lambda_{k}=\mathbb{Z}[x_1, \ldots, x_k]^{\mathfrak{S}_k}$ be the ring of symmetric functions of $k$ variables and $(\Lambda_k)_{\mathbb{Q}}=\Lambda_m\otimes_{\mathbb{Z}}\mathbb{Q}$. It is well-known that the \textit{power sum symmetric functions} \cite[$\mathrm{P}_{23}$]{Macdonald}
\begin{equation*}
  p_0(\bx)=1\text{ and } p_a(\bx)=\sum_{i=1}^{k}x_{i}^a\text{ for } a=1,2,\cdots
\end{equation*}
generate $\Lambda_{k}$ under the power series multiplication, whilst $\Lambda_k^n$ has a basis
\begin{equation*}
  p_{\mu}(\bx):=p_{\mu_1}(\bx)p_{\mu_2}(\bx)\cdots \text{ with }\mu\vdash n.
\end{equation*}

An alternative basis of $\Lambda_k^n$ is given by the \textit{Schur functions} $S_{\lambda}(\bx)$, $\lambda\vdash n$, which can be defined as
\begin{equation*}
    S_{\lambda}(\bx)=\sum_{\mu\vdash n}Z_{\mu}^{-1}\chi_1^{\lambda}(\mu)p_{\mu}(\bx),
\end{equation*}
here $Z_{\mu}$ is the order of the centralizer in $\mathfrak{S}_n$ of an element of cycle type $\mu$.
Let us remark that $S_{\lambda}(\bx)=0$ whenever $\ell(\lambda)>k$ (see e.g. \cite[Equ.~(3.8)]{King}).\end{point}

\begin{point}{}*
Now we denote by $\Lambda_{k,\ell}$ the ring of polynomials in $x_1, \ldots, x_k,y_1, \ldots, y_{\ell}$, which are separately symmetric in $\mathbf{x}$'s and $\mathbf{y}$'s, namely
\begin{equation*}
  \Lambda_{k,\ell}=\mathbb{Q}[x_1, \ldots, x_k]^{\mathfrak{S}_k}\otimes_{\mathbb{Q}}\mathbb{Q}[y_1, \ldots, y_\ell]^{\mathfrak{S}_\ell}.
\end{equation*}

Following \cite[\S4]{King}, we define the  \textit{power sum supersymmetric function} $p_i(\bx/\by)$ to be \begin{eqnarray*}
       &&p_0(\bx/\by)=1, \\
        &&p_i(\bx/\by)=p_i(\bx)-p_i(\by) \text{ for } i\geq 1.
      \end{eqnarray*}

Recall that $\varsigma$ is a fixed primitive $m$-th root of unity. For $\bmu=(\mu^{(1)}, \ldots, \mu^{(m)})\in\mpn$, we define the power sum supersymmetric function $P_{\bmu}(\bx/\by)$ associated to $\bmu$ as follows. For each integer $t\geq 1$ and $i$ such that $1\leq i\leq m$, set
\begin{equation*}
  P_t^{(i)}(\bx/\by)=\sum_{j=1}^m\varsigma^{-ij}p_{t}(\bx^{(j)}/\by^{(j)}).
\end{equation*}
For a partition $\mu^{(i)}=(\mu_1^{(i)}, \mu_2^{(i)}, \cdots)$, we define a function $P_{\mu^{(i)}}(\bx/\by)\in \Lambda_{k,\ell}$ by
\begin{equation*}
  P_{\mu^{(i)}}(\bx/\by)=\prod_{j=1}^{\ell(\mu^{(i)})}P_{\mu_j^{(i)}}^{(i)}(\bx/\by),
\end{equation*}
and define 
\begin{equation*}
  P_{\bmu}(\bx/\by)=\prod_{i=1}^mP_{\mu^{(i)}}(\bx/\by).
\end{equation*}
\end{point}



\begin{point}{}*
Following  \cite[\S6]{B-Regev} or \cite[$\mathrm{P}_{280}$]{King}, the \textit{supersymmetric Schur function} $S_{{\lambda}^{(i)}}(\bx/\by)\in \Lambda_{k,\ell}$  corresponding to a partition $\lambda^{(i)}=\left(\lambda_1^{(i)}, \lambda_2^{(i)}, \ldots\right)$ is defined as
\begin{equation*}
  S_{\lambda^{(i)}}(\bx^{(i)}/\by^{(i)}):=\sum_{\mu\subset\lambda}(-1)^{|\lambda-\mu|}
  S_{\mu}(\bx^{(i)})S_{\lambda'/\mu'}(\by^{(i)}),
\end{equation*}
where $S_{\lambda'/\mu'}(\by^{(i)})$ is the skew Schur function associated to the conjugate $\lambda'/\mu'$  of the skew partition $\lambda/\mu$. Note that the skew Schur function  $S_{\eta/\theta}(\by^{(i)})$ can be calculated by $S_{\eta/\theta}(\by^{(i)})=\sum_{\nu}c^{\eta}_{\theta\nu}S_{\nu}(\by^{(i)})$, where the coefficients $c^{\eta}_{\theta\nu}$ are determined by the Littlewood-Richardson rule in the product of Schur functions (see \cite[$\mathrm{P}_{143}$]{Macdonald}).

For  $\bgl=(\lambda^{(1)}; \ldots; \lambda^{(m)})\in\mpn$, we define the supersymmetric Schur function associated with $\bgl$ by
\begin{eqnarray*}\label{Equ:Supersymm-multi}
S_{\bgl}(\mathbf{x}/\mathbf{y})&=&\prod_{i=1}^{m}S_{{\lambda}^{(i)}}^{(i)}(\bx^{(i)}/\by^{(i)}),
\end{eqnarray*}
which is a super analogue of  the Schur function associated to $m$-multipartitions defined in \cite[(6.1.2)]{S}. It should be noted that $S_{\lambda^{(i)}}(\bx^{(i)}/\by^{(i)})=0$ unless $\lambda^{(i)}$ is a $(k,\ell)$-hook partition (see e.g. \cite[Equ.~(4.10)]{King}), which implies that $S_{\bgl}(\mathbf{x}/\mathbf{y})=0$ unless $\bgl$ is a $(\bk,\bl)$-hook multipartition.
\end{point}

\begin{point}{}*
Now we are ready to give a combinatorial interpretation of the supersymmetric Schur functions in terms of $(\bk,\bl)$-semistandard tableau, which also shows $S_{\bgl}(\bx/\by)=0$ unless $\bgl$ is $(\bk,\bl)$-hook multipartition of $n$.

Let $\lambda$ be a partition of $n$. A \textit{$(k,\ell)$-tableau} $\ft$ of shape $\lambda$ is obtained from the diagram of $\lambda$ by replacing each node by one of the variables $\mathbf{z}$, allowing repetitions. For a  $(k,\ell)$-tableau $\ft$,  we denote by $\ft_\mathbf{x}$ (resp. $\ft_\mathbf{y}$) the nodes filled with variables $\mathbf{x}$ (resp. $\mathbf{y}$).
We say a  $(k,\ell)$-tableau $\ft$ is \textit{$(k,\ell)$-semistandard} if
\begin{enumerate}
  \item[(a)] $\ft_\mathbf{x}$ (resp. $\ft_\mathbf{y}$) is a tableau (resp. skew tableau); and
  \item[(b)] The entries of $\ft_\mathbf{x}$ are nondecreasing in rows, strictly increasing in columns, that is, $\ft_\mathbf{x}$ is a \textit{row-semistandard tableau}; and
  \item[(c)] The entries of $\ft_\mathbf{y}$ are nondecreasing in columns, strictly increasing in rows, that is, $\ft_\mathbf{y}$ is a \textit{column-semistandard skew tableau}.
\end{enumerate}
For $\bgl\in \mpn$, the \textit{($\bk,\bl$)-semistandard tableau} $\ft$ of shape $\bgl$ is a filling of boxes of $\bgl$ with variables  $\mathbf{z}$ such that its $i$-component $\ft^{(i)}$  is a $(k_i,\ell_i)$-semistandard tableau filled variables $\bx^{(i)}$ and $\by^{(i)}$ for $i=1, \ldots, m$.


We denote by $\mathrm{std}_{\bk|\bl}(\bgl)$ the set  of $(\bk,\bl)$-semistandard tableaux of shape $\bgl$ and by  $s_{\bk|\bl}(\bgl)$ its cardinality. Thanks to \cite[\S\,2]{B-Regev} and \cite[Lemma~4.2]{BKK}, $s_{\bk|\bl}(\bgl)\neq 0$ if and only if $\bgl\in H(\bk|\bl;m,n)$.

For a $(\bk,\bl)$-semistandard tableau $\ft$ of shape $\bgl$,  we define
\begin{equation}\label{Equ:Schur=tableau}
  \ft(\bx/\by)=\prod_{i=1}^{m}\prod_{(a,b)\in\lambda^{(i)}}(-1)^{\overline{z}_{a,b}}z_{a,b}
\end{equation}
where $z_{a,b}\in \mathbf{z}$ is the variable filled in the box  $(a,b)$ of $\ft^{(i)}$. It follows from \cite{B-Regev} and \cite[Equ.~(4.9)]{King} that for $\bgl\in H(\bk|\bl;m,n)$, we have
\begin{equation*}
  S_{\bgl}(\bx/\by)=\sum_{\ft\in\mathrm{std}_{\bk|\bl}(\bgl)}\ft(\bx/\by).
\end{equation*}
\end{point}

The following fact clarifies the relationship between supersymmetric Schur functions and power sum supersymmetric functions indexed by multipartitions, which is a cyclotomic version of \cite[(4.4)]{King} and may serve as a definition of supersymmetric Schur functions associated to (hook)mutlipartitions.
\begin{proposition}[\protect{\cite[Proposition~3.8]{Z-Frob}}]\label{Prop:super-Schur-power}Let $Z_{\bmu}$ be the order of the centralizer in $W_{m,n}$ of an element of cycle type $\bmu$. Then
  \begin{equation*}
    S_{\bgl}(\bx/\by)=\sum_{\bmu\in\mpn}Z_{\bmu}^{-1}\chi_1^{\bgl}(\bmu)P_{\bmu}(\bx/\by).
  \end{equation*}
  \end{proposition}

Now we are going to define a deformation of power sum supersymmetric functions indexed by multipartitions. We need the following notations.

 For positive integer $t$, we let \begin{eqnarray*}
&&\mathscr{I}(t;k|\ell)=\{\bi=(i_1, \ldots, i_t)|1\leq i_a\leq k+\ell,a=1,\ldots,t\},\\
          &&\mathscr{I}_{\leq}(t;k|\ell)=\{\bi=(i_1, \ldots, i_t)|1\leq i_1\leq i_2\leq\cdots\leq i_t\leq k+\ell\}.
    \end{eqnarray*}
For $i=1, \ldots, k+\ell$, we define its color $c(i)=a$ whenever $d_{a-1}< i\leq d_a$ for some $1\leq a\leq m$. Given a sequence $\mathbf{i}=(i_1, \ldots, i_t)$ of length $\ell(\mathbf{i})=t$, we define its color $c(\mathbf{i})$ to be the color of $\max\{i_1, \ldots, i_t\}$ and write $\mathbf{z}_{\mathbf{i}}=(-1)^{\bar{\mathbf{i}}}z_{i_1}z_{i_2}\cdots z_{i_t}$. Further,  for $a=0,1$, we denote by $\ell_a(\mathbf{i})$ (resp. $\ell_a(\mathbf{i}_<)$) the number of $i_b$  (resp. $i_b<i_{b+1}$) in $\mathbf{i}$ with parity $\bar{a}$.

Now we can define a deformation of $P_{\bmu}(\bx/\by)$. For integer $t\geq 1$ and $a=1, \ldots, m$, let
\begin{equation*}\label{Equ:q-t}
 q_{t}^{(a)}(\bx/\by;q,\mathbf{Q})=\sum_{\mathbf{i}\in\mathscr{I}_{\leq}(t;k|\ell)}Q^a_{c(\mathbf{i})}
 (-q^{-1})^{\ell_1(\mathbf{i})-\ell_1(\mathbf{i}_<)}q^{\ell_0(\mathbf{i})-\ell_0(\mathbf{i}_<)}
 (q-q^{-1})^{\ell(\mathbf{i}_<)-1}\mathbf{z}_{\mathbf{i}}.
\end{equation*}

For $\bmu=(\mu^{(1)}, \ldots, \mu^{(m)})\in\mpn$, we define
\begin{equation}\label{Equ:Def-q-mu}
  q_{\bmu}(\bx/\by;q,\mathbf{Q})=\sum_{i=1}^m\prod_{j=1}^{\ell(\mu^{(i)})}q^{(i)}_{\mu_j^{(i)}}(\bx/\by;
  q,\mathbf{Q}).
\end{equation}
Note that $q_{\bmu}(\bx/\by;1,\boldsymbol{\varsigma})=P_{\bmu}(\bx/\by)$ with $\boldsymbol{\varsigma}=(\varsigma,\varsigma^2,\ldots, \varsigma^m)$.

In \cite[Theorem~4.11]{Z-Frob}, we obtain the following super Frobenius formula for characters of $H$,
\begin{equation}\label{Equ:char-Schur}
    q_{\bmu}(\bx/\by;q,\bq)=\sum_{\bgl\in\mpn}\chi_q^{\bgl}(\bmu)S_{\bgl}(\bx/\by),
  \end{equation}
  for each $\bmu\in\mpn$.

\begin{point}{}*\label{Subsec:identity}The remain of the section devotes to clarify the relationship between $q_{t}^{(a)}(\bx/\by;q,\mathbf{Q})$ and the super Hall-Liitlewood functions. For integers $a,b$, we denote by $a^b$ the sequence $(a,\ldots, a)$ of length $b$. Then $\bi=(i_1, \ldots, i_t)\in \mathscr{I}_{\leq}(t;k|\ell)$ may be written uniquely as
\begin{equation*}\label{Equ:i=(alpha;beta)}
 \bi(\alpha;\beta)\!\!=\!\!\left(\!1^{\alpha^{(\!1\!)}_1}\!\cdots k_1^{\alpha^{(\!1\!)}_{k_1}}(\!k_1\!\!+\!\!1\!)^{\beta^{(\!1\!)}_1}\!\cdots d_1^{\beta^{(\!1\!)}_{\ell_1}}\cdots\cdots(\!d_{m\!-\!1}\!\!+\!\!1\!)^{\alpha^{(\!m\!)}_1}\!\cdots (\!d_{m}\!\!-\!\!\ell_{m}\!)^{\alpha^{(\!m\!)}_{k_m}}
 (\!d_{m\!-\!1}\!\!+\!\!1\!)^{\beta^{(\!m\!)}_1}\!\cdots d_m^{\beta^{(m)}_{\ell_m}} \!\right)
\end{equation*}for some $(\alpha;\beta)\in\mathscr{C}(t;k|\ell)$, where $\alpha=(\alpha^{(1)};\ldots;\alpha^{(m)})$, $\beta=(\beta^{(1)};\ldots;\beta^{(m)})$ and $\mathscr{C}(t;k+\ell)$ is the set of compositions of $t$ with at most $(k+\ell)$-parts, that is,
 \begin{align*}
  \mathscr{C}(t;k+\ell)=\left.\left\{(\alpha;\beta)=\left(\alpha^{(1)}_1, \ldots, \alpha^{(m)}_{k_m};\beta^{(1)}_{1},\cdots, \beta_{\ell_m}^{(m)}\right)\right||\alpha|+|\beta|=t\right\}.
 \end{align*}
 Thus we may and will identify $\mathscr{I}^{+}(t;k|\ell)$ with $\mathscr{C}(t;k+\ell)$ and write $\mathbf{i}=(\alpha;\beta)$ when $\mathbf{i}=\mathbf{i}(\alpha;\beta)$. Clearly, for $\mathbf{i}=(\alpha;\beta)$, we have $\ell_0(\mathbf{i})=|\alpha|$, $\ell_0(\mathbf{i}_<)=\ell(\alpha)$, $\ell_1(\mathbf{i})=|\beta|$, $\ell_1(\mathbf{i}_<)=\ell(\beta)$, $\mathbf{z}_{\mathbf{i}}=\mathbf{x}^a(-\mathbf{y})^{\beta}$, $c(\mathbf{i})$ is the maximal $b$ with $(\alpha^{(b)},\beta^{(b)})\neq \emptyset$ and we write $\mathbf{Q}_{c(\mathbf{i})}=\mathbf{Q}_{(\alpha;\beta)}$. Therefore
 \begin{equation*}
 q_{t}^{(a)}(\bx/\by;q,\mathbf{Q})=\sum_{(\alpha;\beta)\in\mathscr{C}(t;k|\ell)}Q^a_{(\alpha;\beta)}
  \tilde{q}_{(\alpha;\beta)}(\bx/\by;q),
\end{equation*}
where \begin{equation*}\label{Equ:tilde-q}
  \tilde{q}_{(\alpha;\beta)}(\bx/\by;q)=(\!-q^{-1})^{|\beta|-\ell(\beta)}q^{|\alpha|-\ell(\alpha)}
   (q\!-\!q^{-\!1})^{\ell(\alpha;\beta)-1}\bx^{\alpha}(\!-\by)^{\beta}.
\end{equation*}

In  \cite[Equ.~(4.7)]{Z-Frob}, we show
\begin{equation}\label{Equ:q-q}
  \sum_{(\alpha;\beta)\in\mathscr{C}(t;k|\ell)}\tilde{q}_{(\alpha;\beta)}(\bx/\by;q)=
  \frac{q^t}{q-q^{-1}}q_t(\bx/\by;q^{-2}).
\end{equation}
where $q_t(\bx/\by;q^{-2})$ is the super Hall-Littlewood function (see \cite[\S2]{M2010}).

For $\mathbf{c}=(c_1,\ldots, c_m)\in\mathscr{C}(t;m)$, we let $\mathbf{Q}_{\mathbf{c}}=Q_b$ if $b$ is the largest number $b$ such that $c_b\neq 0$ and write $(\alpha;\beta)\in \mathbf{c}$ if $(\alpha;\beta)\in\mathscr{C}(t;k+\ell)$ satisfies $|\alpha^{(i)}|+|\beta^{(i)}|=c_i$ for $i=1, \ldots, m$. In \cite[Proposition~4.9]{Z-Frob}, we show the function $q_{t}^{(a)}(\bx/\by;q,\mathbf{Q})$ can be expressed as following
\begin{equation*}
  q_t^{(a)}(\bx/\by;q,\bq)=\frac{q^t}{q-q^{-1}}\sum_{(\alpha;\beta)\in\mathscr{C}(t;k+\ell)}
  Q_{\bc}^a\prod_{j=1}^mq_{c_j}\left(\bx^{(j)}/\by^{(j)};q^{-2}\right).
\end{equation*}
 \end{point}
\section{The combinatorial formula of $q_{\bmu}(\bx/\by;q,\bq)$}\label{Sec:mu-weight-i}
In this section, we define the $\bmu$-weight of parity-integer sequences  and interpret the function $q_{\bmu}(\bx/\by;q,\bq)$ as a sum over $\bmu$-weighted parity-integer sequences.

\begin{definition}For $\mathbf{i}=(i_1,\ldots, i_t)\in \mathscr{I}(t;k|\ell)$, we say that $\mathbf{i}$ is \textit{up-down}
 if there exists a $p$ such that  \begin{eqnarray*}i_1<\cdots<i_p\geq i_{p+1}\geq \cdots\geq i_t\text{ for some} \text{ with }0\leq a<t.\end{eqnarray*}We say that $i_{p+1}$ is the \textit{peak} of the up-down sequence $\mathbf{i}$ and write $p(\mathbf{i})=p$.
\end{definition}

Given $\mathbf{i}=(i_1,\ldots, i_t)\in \mathscr{I}(t;k|\ell)$, we denote by $\ell_a(\mathbf{i})$ the number of $i_j$ in $\mathbf{i}$ with parity $\bar{a}$ for $a=0,1$ and write $\ell(\mathbf{i})$ the length of $\mathbf{i}$. Clearly we have $\ell_0(\mathbf{i})+\ell_1(\mathbf{i})=\ell(\mathbf{i})$.
For a up-down sequence $\mathbf{i}=(i_1, \ldots,i_t)$ with $p(\mathbf{i})=p$, we let $\ell_{a}(\mathbf{i}_<)$ be the number of $i_j$ in $\{i_1, \ldots, i_p\}$ with parity $\bar{a}$ for $a=0,1$.
\begin{definition}\label{Def:weight}For a sequence $\mathbf{i}=(i_1, \ldots,i_t)\in \mathscr{I}(t;k|\ell)$, we define its weight\begin{equation*}\mathrm{wt}(\mathbf{i})\!=\!\left\{\begin{array}{ll}
0, & \hbox{if } \mathbf{i}\text{ is not up-down}; \\
(\!-q^{-\!1}\!)^{\ell_1(\mathbf{i})+\ell_0(\mathbf{i}_<)-\ell_0(\mathbf{i}_<)}
q^{\ell_0(\mathbf{i})+\ell_1(\mathbf{i}_<)-\ell_0(\mathbf{i}_<)-1}, & \hbox{if } \mathbf{i} \text{ is up-down with } p(\mathbf{i})=a.\end{array}\right.
\end{equation*}
\end{definition}
Note that if $\ell(\mathbf{i})=1$ then $\mathrm{wt}(\mathbf{i})=1$ when $\bar{\mathbf{i}}=\bar{0}$;  otherwise $\mathrm{wt}(\mathbf{i})=-q^{-2}$.

\begin{remark}If $\ell=0$ then Definition~\ref{Def:weight} reduces to the one defined in Ram's work \cite[Lemma~1.5]{Ram98} (see also \cite[Equ.~(2.13)]{CHM}).
\end{remark}

\begin{lemma}\label{Lemm:wt-seq}
  Given $\mathbf{i}=(i_1,\ldots,i_t)\in \mathscr{I}_{\leq}(t;k|\ell)$ and let $\mathfrak{S}_{\mathbf{i}}$ denote the set of all distinct permutations of $\mathbf{i}$. Then
  \begin{equation*}
    (-q)^{\ell_1(\mathbf{i}_{<})-\ell_1(\mathbf{i}_{=})}q^{\ell_0(\mathbf{i}_{=})-\ell_0(\mathbf{i}_{<})}
   (q\!-\!q^{-\!1})^{\ell(\mathbf{i})-1}=\sum_{\sigma\in\mathfrak{S}_{\mathbf{i}}}
    \mathrm{wt}(\sigma(\mathbf{i})),
  \end{equation*}
  where $\ell_s(\mathbf{i}_{=})$ (resp. $\ell_s(\mathbf{i}_{<})$) is the number of $i_j\in\mathbf{i}$ with parity $\bar{s}$ such that $i_j<i_{j+1}$ (resp. $i_j=i_{j+1}$), $s=0,1$.
\end{lemma}
\begin{proof}We adapt the argument of the proof of \cite[Lemma~1.5]{Ram98}. Let $\tilde{D}$ be the set of distinct elements in the sequence $\mathbf{i}$ and let $D=\tilde{D}-\{i_t\}$. For $s=0,1$, let $D_s$ be the subset of $D$ consisting of elements with parity $\bar{s}$. For each subset $A$ of $D$, we have a disjoint decomposition $A=A_0\cup A_1$ with $A_s\subseteq D_s$ for $s=0,1$. We let  $a_s=|A_s|$ and $a=|A|$, $j_1<j_1<\cdots<j_a$ the elements of $A$ in increasing order (clearly $j_a=i_t$) and let $j_{a+2}\geq j_{a+3}\geq \cdots\geq j_{t}$ be the remainder of the elements of the sequence $\mathbf{i}$ arranged in decreasing order. In this way, we obtain a up-down sequence
\begin{equation*}
 \mathbf{j}=(j_1<\ldots<j_a\geq j_{a+1}\geq j_{a+2}\geq j_{a+3}\geq \cdots\geq j_{t})
\end{equation*}
 with $p(\mathbf{j})=a$. According to Definition~\ref{Def:weight},
 \begin{equation*}
   \mathrm{wt}(\mathbf{j})=(-q)^{a_1-\ell_1(\mathbf{i})-a_0}q^{\ell_0(\mathbf{i})+a_1-1-a_0}
 \end{equation*}

Note that for $\sigma\in \mathfrak{S}_{\mathbf{i}}$, $\mathrm{wt}(\sigma(\mathbf{i}))\neq 0$ if and only if there exists an $a$ such that $0\leq a< t$ and
  \begin{equation*}
    i_{\sigma(1)}<\cdots<i_{\sigma(a)}<i_{\sigma(a+1)}\geq i_{\sigma(a+2)}\geq \cdots\geq i_{\sigma(t)}.
  \end{equation*}

It follows from this that every $\sigma\in \mathfrak{S}_{\mathbf{i}}$ such that $\mathrm{wt}(\sigma(\mathbf{i}))\neq 0$ is given by a unique subset $A$ of $\tilde{D}$. Thus
\begin{eqnarray*}
\sum_{\sigma\in\mathfrak{S}_{\mathbf{i}}}
    \mathrm{wt}(\sigma(\mathbf{i}))&=&\sum_{A_0\subseteq D_0,A_1\subseteq D_1}(-q)^{a_1-\ell_1(\mathbf{i})-a_0}q^{\ell_0(\mathbf{i})+a_1-1-a_0}  \\
   &=& \sum_{a_0=0}^{\ell_0(\mathbf{i}_<)}\sum_{a_1=0}^{\ell_1(\mathbf{i}_<)}
   \tbinom{\ell_0(\mathbf{i}_<)}{a_0}\tbinom{\ell_1(\mathbf{i}_<)}{a_1}
  (-q)^{a_1-\ell_1(\mathbf{i})-a_0}q^{\ell_0(\mathbf{i})+a_1-1-a_0}\\
  &=& \sum_{a_0=0}^{\ell_0(\mathbf{i}_<)}
  \tbinom{\ell_0(\mathbf{i}_<)}{a_0}(-q)^{a_0}q^{\ell_0(\mathbf{i})-1-a_0}
  \biggl(\sum_{a_1=0}^{\ell_1(\mathbf{i}_<)}
   \tbinom{\ell_1(\mathbf{i}_<)}{a_1}
  q^{a_1}(-q)^{a_1-\ell_1(\mathbf{i})}\biggr)\\
  &=& (-q)^{\ell_1(\mathbf{i}_<)-\ell_1(\mathbf{i})}(q-q^{-1})^{\ell_1(\mathbf{i}_<)}
  \sum_{a_0=0}^{\ell_0(\mathbf{i}_<)}
  \tbinom{\ell_0(\mathbf{i}_<)}{a_0}(-q)^{a_0}q^{\ell_0(\mathbf{i})-1-a_0}\\
  &=&(-q)^{\ell_1(\mathbf{i}_{<})-\ell_1(\mathbf{i})}q^{\ell_0(\mathbf{i})-\ell_0(\mathbf{i}_{<})}
   (q\!-\!q^{-\!1})^{\ell(\mathbf{i})-1}.
\end{eqnarray*}
It completes the proof.
\end{proof}

Note that $\mathrm{wt}(\mathbf{i})$ is zero unless $\mathbf{i}$ is up-down, i.e., $\mathbf{i}=(\alpha;\beta)$ with $(\alpha;\beta)\in \mathscr{C}(t;k+\ell)$ for some $t$ (see \S\ref{Subsec:identity}). Thanks to Equ.~(\ref{Equ:tilde-q}) and Lemma~\ref{Lemm:wt-seq}, we obtain
\begin{equation*}
    \mathrm{wt}(\alpha;\beta):=(\!-q^{-1})^{|\beta|-\ell(\beta)}q^{|\alpha|-\ell(\alpha)}
   (q\!-\!q^{-\!1})^{\ell(\alpha;\beta)-1}=\sum_{\sigma\in\mathfrak{S}_{\mathbf{i}}}
    \mathrm{wt}(\sigma(\mathbf{i})).
  \end{equation*}
or equivalently,
\begin{equation*}\label{Equ:weight=partition}
 \sum_{\sigma\in\mathfrak{S}_{\mathbf{i}}}
    \mathrm{wt}(\sigma(\mathbf{i}))\mathbf{z}_{\sigma(\mathbf{i})}=
    \tilde{q}_{(\alpha;\beta)}(\bx/\by;q),
\end{equation*}
due to $\mathbf{z}_{\sigma(\mathbf{i})}=\mathbf{z}_{\mathbf{i}}=\bx^{\alpha}(-\by)^{\beta}$ for all $\sigma\in\mathfrak{S}_{\mathbf{i}}$.

As a consequence, we can rephrase the function $q_t^{(a)}(\bx/\by;q,\bq)$ as following
\begin{corollary}\label{Cor:q-wt}
  For integers $t\geq 1$ and $1\leq a\leq m$, we have
  \begin{eqnarray*}q^{(a)}_t(\bx/\by;q,\bq) &=&\sum_{(\alpha;\beta)\in\mathscr{C}(t;k+\ell)}\mathrm{wt}(\alpha;\beta)Q_{(\alpha;\beta)}^a \bx^{\alpha}(-\by)^{\beta}.\end{eqnarray*}
\end{corollary}
\begin{proof}Note that $Q_{\mathbf{i}}=Q_{\sigma(\mathbf{i})}$  for all $\sigma\in \mathfrak{S}_{\mathbf{i}}$ and $\mathrm{wt}(\mathbf{i})=0$ unless $\mathbf{i}$ is up-down. Therefore
\begin{eqnarray*}
  \sum_{\mathbf{i}=(i_1,\ldots,i_t)}\mathrm{wt}(\mathbf{i})Q_{c(\mathbf{i})}^a\mathbf{z}_{\mathbf{i}} &=&\sum_{\mathbf{i}\in \mathscr{I}_{\leq}(t;k|\ell)}\sum_{\sigma\in\mathfrak{S}_{\mathbf{i}}}
  \mathrm{wt}(\sigma(\mathbf{i}))Q_{c(\sigma(\mathbf{i}))}^a\mathbf{z}_{\sigma(\mathbf{i})}\\
   &=&\sum_{\mathbf{i}\in \mathscr{I}_{\leq}(t;k|\ell)}Q_{c(\mathbf{i})}^a
  \left(\sum_{\sigma\in\mathfrak{S}_{\mathbf{i}}}\mathrm{wt}(\sigma(\mathbf{i}))\right)\mathbf{z}_{\mathbf{i}}\\
  &=&\sum_{(\alpha;\beta)\in \mathscr{C}(t;k+\ell)}\mathrm{wt}(\alpha;\beta)Q_{(\alpha;\beta)}^a\bx^{\alpha}(-\by)^{\beta}.
 \end{eqnarray*}
 It completes the proof.
\end{proof}

For $\bmu\in\mpn$, we denote by $\mathfrak{t}^{\bmu}$ the standard $\bmu$-tableau  with the numbers $1,2,\dots,n$ entered in order first along the rows (resp. columns) of the first component, and then along the rows (resp. columns) of the second component, and so on. For $i=1, \ldots, n$, we will write $c_{\mathfrak{t}^{\bmu}}(i)=c$ if the number $i$ is entered in the $c$-component of $\mathfrak{t}^{\bmu}$. More generally, given a standard tableau $T$ filling the numbers $1,2,\dots,n$  of shape $\bgl\in\mpn$, we write $c_{T}(i)=c$ when $i$ is entered in the $c$-component of $T$.

\begin{definition}\label{Def:mu-up-down}We say that $\mathbf{i}=(i_1, \ldots, i_n)$ is \textit{$\bmu$-up-down} if it satisfies the following property: if $a, a+1, \ldots, a+r$ is a row of $\mathfrak{t}^{\bmu}$ then the subsequence $(i_{a}, i_{a+1}, \ldots, i_{a+r})$ is up-down. For a $\bmu$-up-down sequence $\mathbf{i}$, we denote by $P_{\mathbf{i}}^{\bmu}$ the set of peaks in $\mathbf{\mathbf{i}}$, one for each row of $\mathfrak{t}^{\bmu}$.\end{definition}

\begin{definition}\label{Def:mu-weight}For a sequence $\mathbf{i}=(i_1, \ldots, i_n)\in \mathscr{I}(n;k|\ell)$, we define its $\bmu$-weight\begin{eqnarray*}\mathrm{wt}_{\bmu}(\mathbf{i})&=&\left\{\begin{array}{ll}
0, & \hbox{if } \mathbf{i}\text{ is not $\bmu$-up-down}; \\
(-q^{-1})^{\ell_1(\mathbf{i}_{\geq})+\ell_0(\mathbf{i}_<)}q^{\ell_0(\mathbf{i}_{\geq})+\ell_1(\mathbf{i}_<)}
\prod_{i\in P_{\mathbf{i}}^{\bmu}}Q_{c(i)}^{c_{\mathfrak{t}^{\bmu}}(i)}, & \hbox{if } \mathbf{i} \text{ is $\bmu$-up-down},\end{array}\right.
\end{eqnarray*}
where $\ell_a(\mathbf{i}_{\geq})$ (resp. $\ell_a(\mathbf{i}_{<})$) is the number of $i_j\geq i_{j+1}$ (resp. $i_j < i_{j+1}$) in the same row of $\mathfrak{t}^{\bmu}$ with parity $\bar{a}$ for $a=0,1$.
\end{definition}
\begin{example}\label{Exam:mu-weight}Let $n=20$, $m=3$, $\bk|\bl=(1|1,1|2,1|3)$ and    $\bmu=\left((2,1,1);(3,2,2,1);(4,3,1)\right)$. Then
\begin{equation*}
  \mathfrak{t}^{\bmu}=\biggl(\,\Tricolor(1&2\cr 3\cr 4 |5&6&7\cr 8&9
  \cr 10&11\cr 12|13&14&15&16
  \cr 17&18&19\cr 20)\,\biggr),
\end{equation*}
the  (row-reading) sequence
\begin{equation*}
  \mathbf{i}=\biggl(\,\Tricolor(1&\underline{3}\cr \underline{2}\cr \underline{4} |6&7&\underline{9}\cr \underline{2}&2
  \cr \underline{5}&4\cr 7|8&6&5&\underline{7}
  \cr 3&4&\underline{6}\cr \underline{8})\,\biggr).
\end{equation*}
 is $\bmu$-up-down and its peaks are underlined numbers. Furthermore, we have
 \begin{eqnarray*}
  \mathrm{wt}_{\bmu}(\mathbf{i})&=&{\biggl(\,\Triwt({(\!-q\!)^{-\!1}}&Q_2\cr Q_1\cr Q_2 |(\!-q\!)^{-\!1}&q&Q_3^2\cr Q_1^2&(\!-q\!)^{-\!1}
  \cr Q_2^2&(\!-q\!)^{-\!1}\cr Q_3^2|{(\!-q\!)^{-\!1}}&q&q&Q_3^3
  \cr {(\!-q\!)^{-\!1}} &q&Q_3^3\cr Q_3^3)\,\biggr)}\\
    &=&q^{-2}Q_1^3Q_2^4Q_3^{13}.
\end{eqnarray*}
\end{example}
A point should be noted is that Definition~(\ref{Equ:Def-q-mu}) of $q_{\bmu}$ can be thought of as product of $q^{(i)}_t$ over the rows of $\mathfrak{t}^{\bmu}$ where $t$ are the length of the rows in the $i$-component $\mathfrak{t}^{\bmu}$. Now we can rewrite the super Frobenius formula for the characters of $H$ as a sum over $\bmu$-weighted parity-integer sequences, which follows directly from Corollary~\ref{Cor:q-wt}.
\begin{corollary}\label{Cor:q-mu-wt}For $\bmu\in\mpn$,
  \begin{equation*}
    q_{\bmu}(\bx/\by;q,\mathbf{Q})=\sum_{\mathbf{i} \text{ is $\bmu$-up-down}}\mathrm{wt}_{\bmu}(\mathbf{i})\mathbf{z}_{\mathrm{i}}.
  \end{equation*}
\end{corollary}
\section{RSK superinsertion}\label{Sec:RSK}

In the section, we first review the RSK superinsertion algorithm for partitions, then extend it to work for multipartitions and defines the $\bmu$-weight of standard tableaux, which leads to a new proof of the super Frobenius formula.
\begin{point}{}*
Let
\begin{equation*}
  \mathcal{Z}_{k,\ell}(n)=\left\{z_{\mathbf{j}}=\left.\left(\begin{array}{ccc}1&\cdots& n\\
 z_{j_1}& \cdots&z_{j_n}\\ \end{array} \right)\right|1\leq j_1, \ldots, j_n\leq k+\ell\right\}.
\end{equation*}
Following Berele-Regev \cite[\S2]{B-Regev} (see also \cite[\S4]{DR}), the \textit{RSK  superinsertion} is a bijection from $\mathcal{Z}_{k,\ell}(n)$ to pairs of tableaux $(S,T)$, where $S$ is a $(k,\ell)$-semistandard tableau, $T$ is a standard tableau (which is called the \textit{recording tableau}), and $\mathrm{shape}(S)=\mathrm{shape}(T)=\gl$ for some partition $\lambda$ of $n$. More precisely, for any $z_{\mathbf{j}}\in\mathcal{Z}_{k,\ell}(n)$, the RSK superinsertion constructs the pairs of tableaux $(S,T)=(S(z_{\mathbf{j}}),T(z_{\mathbf{j}}))$ iteratively:
 \begin{equation*}
   (\emptyset,\emptyset)=(S_0,T_0), (S_1,T_1), \ldots, (S_n,T_n)=(S,T)
 \end{equation*}
 by applying the following rules:
\begin{enumerate}
  \item[(1)] $S_i$ is a $(k,\ell)$-semistandard tableau containing $i$ nodes and $T_i$ is a standard tableau with $\mathrm{shape}(T_i)=\mathrm{shape}(S_i)$.

  \item[(2)] $S_i$ is obtained from $S_{i-1}$ by inserting $z_{j_i}$ into $S_{i-1}$ as follows:

  \begin{enumerate}

  \item[(i)] If $z_{j_i}\in\mathbf{x}$ then $P_i$ is obtained from $P_{i-1}$ by column inserting $z_{j_i}$ into $P_{i-1}$ as follows:
           \begin{enumerate}
             \item[(a)] Insert $z_{j_i}$ into the first column of $P_{i-1}$ by displacing the smallest variables $\geq z_{j_i}$; if every variable is $< z_{j_i}$, add $z_{j_i}$ to the bottom of the first column.
             \item[(b)] If $z_{j_i}$ displace $z_{?}$ from the first column, insert $z_{?}$ into the second column using the above rule whenever $z_{?}\in \mathbf{x}$, otherwise insert $z_{?}$ to the first row using the rule (ii).
             \item[(c)] Repeat for each subsequent column, until a variable in $\mathbf{x}$'s  is added to the bottom of some (possible empty) column.
           \end{enumerate}
     \item[(ii)] If $z_{j_i}\in\mathbf{y}$ then $S_{i}$ is obtained from $S_{i-1}$ by row inserting $z_{j_i}$ into $S_{i-1}$ as follows:
           \begin{enumerate}
             \item[(a)] Insert $z_{j_i}$ into the first row of $P_{i-1}$ by displacing the smallest variable $\geq z_{j_i}$; if every variable is $< z_{j_i}$, add $z_{j_i}$ to the right-end of the first row.
             \item[(b)] If $z_{j_i}$ displacing $z_{?}$ from the first row, insert $z_{?}$ into the second row using the above rules (note that here $z_{?}$ must be in $\by$).
             \item[(c)] Repeat for each subsequent row, until a variable in $\mathbf{y}$'s  is added to the right-end of some (possible empty) row.
           \end{enumerate}
         \end{enumerate}
  \item[(3)] $T_i$ is obtained  from $T_{i-1}$ by inserting number $i$ in the newly added box.
\end{enumerate}

 \end{point}

 \begin{example}\label{Exam:1-insertion}The result of inserting  of $\mathbf{z}_{\mathbf{j}}=\left(\begin{array}{ccccccccc}
  1&2&3&4&5&6&7&8&9\\y_1&x_2&x_2&x_1&x_3&y_1&y_1&y_3&y_2
\end{array}\right)$ is
\noindent   \begin{equation*}
 \begin{array}{llllll}
      S:& \diagram(y_1)&{\diagram(x_2& y_1)}&{\diagram(x_2& x_2& y_1)}&
      {\diagram(x_1& x_2&x_2& y_1)}&{\diagram(x_1&x_2& x_2&y_1\cr x_3)}\\  \\
   T:&\diagram(1)&{\diagram(1&2)}&{\diagram(1& 2&3)}&{\diagram(1&2&3&4)}&{\diagram(1&2&3&4\cr 5)}
 \end{array}
 \end{equation*}
 \begin{equation*}
 {\begin{array}{lllll}
&{\diagram(x_1&x_2& x_2&y_1\cr x_3& y_1)}&{\diagram(x_1&x_2& x_2&y_1\cr x_3& y_1\cr y_1)}
&{\diagram(x_1&x_2& x_2&y_1&y_3\cr x_3& y_1\cr y_1)}&{\diagram(x_1&x_2& x_2&y_1&y_2\cr x_3&y_1&y_3\cr y_1)} \\             \\
&{\diagram(1&2&3&4\cr 5& 6)}&{\diagram(1&2&3&4\cr 5&6\cr 7)}&
{\diagram(1&2&3&4&8\cr 5&6\cr 7)}&{\diagram(1&2&3&4&8\cr 5&6&9\cr 7)}.
      \end{array}}
   \end{equation*}
 \end{example}

\begin{point}{}* \label{Subsec:RSK}Now we generalize the RSK superinsertion algorithm to work for multipartitions, that is, a bijection from $\mathcal{Z}_{k,\ell}(n)$ to pairs of tableaux $(S,T)$, where $S$ is a $(\bk,\bl)$-semistandard tableau, the \textit{recording tableau} $T$ is a standard tableau, and $\mathrm{shape}(S)=\mathrm{shape}(T)=\bgl$ for some multipartition $\bgl\in\mpn$. For any $z_{\mathbf{j}}\in\mathcal{Z}_{k,\ell}(n)$, the RSK superinsertion algorithm constructs the pairs of tableaux $(S,T)=(S(z_\mathbf{j}),T(z_\mathbf{j)})$ iteratively:
 \begin{equation*}
   (\emptyset,\emptyset)=(S_0,T_0), (S_1,T_1), \ldots, (S_n,T_n)=(S,T)
 \end{equation*}
 by applying the following rules:
\begin{enumerate}
  \item[(1)] $S_i$ is a $(\bk,\bl)$-semistandard tableau containing $i$ nodes and $T_i$ is a standard tableau with $\mathrm{shape}(T_i)=\mathrm{shape}(S_i)$;

  \item[(2)] $S_i$ is obtained from $S_{i-1}$ by inserting $z_{j_i}$ into $S_{i-1}^{(c)}$, write as $S_i\!=\!S_{i-1}\!\leftarrow\! z_{j_i}$, by applying the RSK superinsertion algorithms when $z_{j_i}$ is being of color $c$;

  \item[(3)] $T_i$ is obtained  from $T_{i-1}$ by inserting number $i$ in the newly added box.
\end{enumerate}
\end{point}
\begin{remark}
  If $\ell=0$ (resp. $\bl=\textbf{0}$)  then the RSK superinsertion reduces to the classical RSK column insertions for partitions (resp. multipartitions) (see e.g. \cite[Chapter~3.1]{Sagan} and  \cite[\S~3]{CHM}). The RSK superinsertion here is different from the $(k,\ell)$-RoSch insertion given in \cite[\S2.5]{B-Regev}
\end{remark}

\begin{example}\label{Exam:2-insertion}The inserting $z_{\mathbf{j}}\!\!=\!\!\!\left(\!\!\begin{array}{cccccccccc}
  1&2&3&4&5&6&7&8&9&10\\x_1^{(\!1\!)}&y_1^{(\!1\!)}&x_1^{(\!2\!)}&x_2^{(\!1\!)}&x_2^{(\!1\!)}
  &x_1^{(\!2\!)}&y_1^{(\!2\!)}
  &y_1^{(\!2\!)}&y_1^{(\!1\!)}&y_2^{(\!1\!)}
\end{array}\!\!\right)$ (with $m=2$) is
\begin{equation*}\small\begin{array}{ccccc}
   S: &\left(\begin{array}{cc}\colordiagram(x_1^{(\!1\!)});&\emptyset\end{array}\right)
&\left(\begin{array}{cc}\colordiagram(x_1^{(\!1\!)}& y_1^{(\!1\!)});&\emptyset\end{array}\right)
&\left(\begin{array}{cc}\colordiagram(x_1^{(\!1\!)}& y_1^{(\!1\!)});&\colordiagram(x_1^{(\!2\!)})\end{array}\right)
&\left(\begin{array}{cc}\colordiagram(x_1^{(\!1\!)}& y_1^{(\!1\!)}\cr x_2^{(\!1\!)});&\colordiagram(x_1^{(\!2\!)})
   \end{array}\right)\\&&&&\\
  T:&\left(\begin{array}{cc}\colordiagram(1);&\emptyset\end{array}\right)
&\left(\begin{array}{cc}\colordiagram(1&2);&\emptyset\end{array}\right)
&\left(\begin{array}{cc}\colordiagram(1&2);&\colordiagram(3)\end{array}\right)
&\left(\begin{array}{cc}\colordiagram(1&2\cr4);&\colordiagram(3)
   \end{array}\right)\end{array}
   \end{equation*}
\begin{equation*}\small\begin{array}{cccc}&\left(\begin{array}{cc}
 \colordiagram(x_1^{(\!1\!)}&x_2^{(\!1\!)}&y_1^{(\!1\!)}\cr x_{2}^{(\!1\!)});&
 \colordiagram(x_1^{(2)})\end{array}\right)&\left(\begin{array}{cc}
   \colordiagram(x_1^{(\!1\!)}&x_2^{(\!1\!)}& y_1^{(\!1\!)}\cr x_{2}^{(\!1\!)});
   &\colordiagram(x_1^{(\!2\!)}&x_1^{(\!2\!)})
   \end{array}\right)&\left(\begin{array}{cc}
   \colordiagram(x_1^{(\!1\!)}&x_2^{(\!1\!)}& y_1^{(\!1\!)}\cr x_{2}^{(\!1\!)});
   &\colordiagram(x_1^{(\!2\!)}&x_1^{(\!2\!)}&y_1^{(\!2\!)})
   \end{array}\right)\\ &&&\\
   &\left(\begin{array}{cc}\colordiagram(1&2&5\cr4);&\colordiagram(3)\end{array}\right)
&\left(\begin{array}{cc}\colordiagram(1&2&5\cr4);&\colordiagram(3&6)\end{array}\right)
&\left(\begin{array}{cc}\colordiagram(1&2&5\cr4);&\colordiagram(3&6&7)\end{array}\right)\end{array}
   \end{equation*}
\begin{equation*}\tiny\small\begin{array}{cccc}&\left(\begin{array}{cc}
 \colordiagram(x_1^{(\!1\!)}&x_2^{(\!1\!)}& y_1^{(\!1\!)}\cr x_{2}^{(\!1\!)});
   &\colordiagram(x_1^{(\!2\!)}&x_1^{(\!2\!)}& y_1^{(\!2\!)}\cr y_1^{(\!2\!)})
   \end{array}\right)&\left(\begin{array}{cc}
   \colordiagram(x_1^{(\!1\!)}&x_2^{(\!1\!)}& y_1^{(\!1\!)}\cr x_{2}^{(\!1\!)}& y_1^{(\!1\!)});
   &\colordiagram(x_1^{(\!2\!)}&x_1^{(\!2\!)}&y_1^{(\!2\!)}\cr y_1^{(\!2\!)})
   \end{array}\right)&\left(\begin{array}{cc}
   \colordiagram(x_1^{(\!1\!)}&x_2^{(\!1\!)}& y_1^{(\!1\!)}&y_2^{(\!1\!)}\cr x_{2}^{(\!1\!)}& y_1^{(\!1\!)});
   &\colordiagram(x_1^{(\!2\!)}&x_1^{(\!2\!)}&y_1^{(\!2\!)}\cr y_1^{(\!2\!)})
   \end{array}\right)\\  &&&\\
    &\left(\begin{array}{cc}
   \colordiagram(1&2&5\cr4);&\colordiagram(3&6& 7\cr 8)\end{array}\right)
 &\left(\begin{array}{cc}\colordiagram(1&2&5\cr4&9);&\colordiagram(3&6&7\cr 8)\end{array}\right)
 &\left(\begin{array}{cc}\colordiagram(1&2&5&10\cr4&9);&\colordiagram(3&6&7\cr 8)\end{array}\right)
 \end{array}
   \end{equation*}
\end{example}
\begin{definition}
Let $\bgl=(\lambda^{(1)}, \ldots, \lambda^{(m)})\in\mpn$ and $T\in\mathrm{std}(\bgl)$. For $a,b\in T$, we define
\begin{eqnarray*}
   &&a\!\xymatrix@C=0.55cm{\ar@{->}[r]^{\substack{\mathrm{SW}}}&}\!b \text{ if }\left\{\begin{array}
     {l}a\in T^{(i)}, b\in T^{(j)} \text{ and } i\!<\!j,  \text{or } b \text{ south (below) and/or west (left) of } a  \text{ in } T^{(j)} \\  \text{ with } \bar{a}=\bar{b}=\bar{0},
     \text{or } b \text{ north (above) and/or east (right) of } a   \text{ in } T^{(j)} \text{ with }\\ \bar{a}=\bar{b}=\bar{1}, \text{ or }  b \text{ south (below) or east (right) of } a \text{ in } T^{(j)} \text{ with }(\bar{a},\bar{b})=(\bar{0},\bar{1}).
   \end{array}\right.\\
  &&a\!\xymatrix@C=0.55cm{\ar@{->}[r]^{\substack{\mathrm{NE}}}&}\!b \text{ if }\left\{\begin{array}
     {l}a\in T^{(i)}, b\in T^{(j)} \text{ and } j<i,  \text{or } b \text{ north (above) and/or east (right) of } a \\ \text{ in } T^{(j)}  \text{ with } \bar{a}=\bar{b}=\bar{0},
     \text{or } b \text{ south (below) and/or west (left) of } a \\  \text{ in } T^{(j)} \text{ with } (\bar{a},\bar{b})=(\bar{1},\bar{1}) \text{ or }(\bar{1},\bar{0}).
   \end{array}\right.\\
\end{eqnarray*}
\end{definition}

\begin{example}Let $\mathbf{z}_{\mathbf{j}}$ be the one in Examples~\ref{Exam:1-insertion} and \ref{Exam:2-insertion} respectively. Then we have
\begin{eqnarray*}
 &&1\!\xymatrix@C=0.55cm{\ar@{->}[r]^{\substack{\mathrm{NE}}}&}\!2\!
 \xymatrix@C=0.55cm{\ar@{->}[r]^{\substack{\mathrm{NE}}}&}\!3
  \!\xymatrix@C=0.55cm{\ar@{->}[r]^{\substack{\mathrm{NE}}}&}\!4
  \!\xymatrix@C=0.55cm{\ar@{->}[r]^{\substack{\mathrm{SW}}}&}\!5
  \!\xymatrix@C=0.55cm{\ar@{->}[r]^{\substack{\mathrm{SW}}}&}\!6
  \!\xymatrix@C=0.55cm{\ar@{->}[r]^{\substack{\mathrm{NE}}}&}\!7
  \!\xymatrix@C=0.55cm{\ar@{->}[r]^{\substack{\mathrm{SW}}}&}\!8
  \!\xymatrix@C=0.55cm{\ar@{->}[r]^{\substack{\mathrm{NE}}}&}\!9 \text{ and }\\&&
  1\!\xymatrix@C=0.55cm{\ar@{->}[r]^{\substack{\mathrm{SW}}}&}\!2\!
  \xymatrix@C=0.55cm{\ar@{->}[r]^{\substack{\mathrm{SW}}}&}\!3
  \!\xymatrix@C=0.55cm{\ar@{->}[r]^{\substack{\mathrm{NE}}}&}\!4\!
  \xymatrix@C=0.55cm{\ar@{->}[r]^{\substack{\mathrm{NE}}}&}\!5
  \!\xymatrix@C=0.55cm{\ar@{->}[r]^{\substack{\mathrm{SW}}}&}\!6
  \!\xymatrix@C=0.55cm{\ar@{->}[r]^{\substack{\mathrm{SW}}}&}\!7
  \!\xymatrix@C=0.55cm{\ar@{->}[r]^{\substack{\mathrm{NE}}}&}\!8
  \!\xymatrix@C=0.55cm{\ar@{->}[r]^{\substack{\mathrm{NE}}}&}\!9
  \!\xymatrix@C=0.55cm{\ar@{->}[r]^{\substack{\mathrm{SW}}}&}\!10.
\end{eqnarray*}
\end{example}

Recall that the linearly order of the variables defined in \S\ref{Point:index-bijection}. Then we have the following fact, which is a super-version of the ones about RSK insertion (see e.g. \cite[Proposition~2.1]{Ram98} or \cite[Proposition~3.1]{CHM}).
\begin{proposition}\label{Prop:SW-NE}
  Let $S_{j+1}=(S_{j-1}\!\!\leftarrow\!\!z_{i_j})\!\!\leftarrow\!\! z_{i_{j+1}}$ with $S_{j-1}$ is $(\bk,\bl)$-semistandard and let $T_{j+1}$ be the associated recording tableau. The following assertions hold: \begin{enumerate}
                     \item[(1)] If $z_{i_j}<z_{i_{j+1}}\in \mathbf{x}$ then $j\!\xymatrix@C=0.55cm{\ar@{->}[r]^{\substack{\mathrm{SW}}}&}\!j+1$ in $T_{j+1}$.
                     \item[(2)] If $z_{i_j}\geq z_{i_{j+1}}\in \mathbf{y}$ then $j\!\xymatrix@C=0.55cm{\ar@{->}[r]^{\substack{\mathrm{NE}}}&}\!j+1$ in $T_{j+1}$.
                       \end{enumerate}
\end{proposition}
\begin{proof}Note that the RSK superinsertion algorithm inserts the variables in $\mathbf{x}$ (resp. $\mathbf{y}$)  by applying classical column (resp. row) RSK insertion algorithm and  $\mathbf{x}^{(a)}<\mathbf{y}^{(a)}<\mathbf{x}^{(a+1)}$ for $a=1,\ldots, m-1$. The two  assertions follows directly by the linearly ordering on variables and well-known facts about the classical column and row RSK insertion algorithms.
\end{proof}

\begin{definition}\label{Def:mu-SW-NE}For $\bgl, \bmu\in \mpn$, we say that $T\in \mathrm{std}(\bgl)$ is \textit{$\bmu$-SW-NE} if it satisfies the property: if $i, i+1, \ldots, i+j$ is in a row of $\mathfrak{t}^{\bmu}$ then
\begin{equation*}
  \begin{array}{l}i\!\xymatrix@C=0.55cm{\ar@{->}[r]^{\substack{\mathrm{SW}}}&}
 \!i+1\!\xymatrix@C=0.55cm{\ar@{->}[r]^{\substack{\mathrm{SW}}}&}
 \!\cdots\!\xymatrix@C=0.55cm{\ar@{->}[r]^{\substack{\mathrm{SW}}}&}
 \!p\!\xymatrix@C=0.55cm{\ar@{->}[r]^{\substack{\mathrm{NE}}}&}
 \!\cdots\!\xymatrix@C=0.55cm{\ar@{->}[r]^{\substack{\mathrm{NE}}}&}\!i+j\text{ in }T.
  \end{array}
\end{equation*}
 The number $p$ here is called the \textit{peak} of the row.\end{definition}

 If $T$ is a $\bmu$-SW-NE tableau, w denote by $P_{T}^{\bmu}$ the set of all peaks in $T$ and let $T_{\mathrm{SW}}$ (resp. $T_{\mathrm{NE}}$) be the set of numbers $j$ such that $j\!\xymatrix@C=0.55cm{\ar@{->}[r]^{\substack{\mathrm{SW}}}&}\!j\!+\!1$ (resp. $j\!\xymatrix@C=0.55cm{\ar@{->}[r]^{\substack{\mathrm{NE}}}&}\!j\!+\!1$) in $T$, one for each row of $\mathfrak{t}^{\bmu}$.

 \begin{definition}\label{Def:tableau-weight}Let $T$ be standard $\bgl$-tableau. We define its $\bmu$-weight\begin{eqnarray*}\mathrm{wt}_{\bmu}(\mathfrak{t})&=&\left\{\begin{array}{ll}
0, & \text{if $\mathfrak{t}$ is not $\bmu$-SW-NE}; \\
\left(-q^{-1}\right)^{\ell_0(T_{\mathrm{SW}})+\ell_1(T_{\mathrm{NE}})}
q^{\ell_0(T_{\mathrm{NE}})+\ell_1(T_\mathrm{SW})}
\displaystyle\prod_{i\in P_{T}^{\bmu}}Q_{c_{T}(i)}^{c_{\mathfrak{t}^{\bmu}}(i)}, & \text{if $\mathfrak{t}$ is  $\bmu$-SW-NE},\end{array}\right.
\end{eqnarray*}
where $\ell_a(T_{\mathrm{SW}})$ (resp. $\ell_a(T_{\mathrm{NE}})$) is the number of elements in  $T_{\mathrm{SW}}$ (resp. $T_{\mathrm{NE}}$) with degree $\bar{a}$ ($a=0,1$).
\end{definition}

\begin{example}\label{Exam:tableau-weight}Let $n=20$, $m=3$, $\bk|\bl=(1|1,1|2,1|3)$ and    $\bmu=\left((2,1,1);(3,2,2,1);(4,3,1)\right)$. Then 
\begin{equation*}
  \mathfrak{t}^{\bmu}=\biggl(\,\Tricolor(1&2\cr 3\cr 4 |5&6&7\cr 8&9
  \cr 10&11\cr 12|13&14&15&16
  \cr 17&18&19\cr 20)\,\biggr)
\end{equation*}
Now we insert the sequence $\mathbf{i}$ of Example~\ref{Exam:mu-weight} via the RSK superinsertion algorithm. Thanks to the bijection in \S\ref{Point:index-bijection} (via the row reading), we get the corresponding variables
\begin{equation*}
  \mathbf{i}=\biggl(\,\Tricolor(x_1^{(\!1\!)}&x_1^{(\!2\!)}\cr y_{1}^{(\!1\!)}\cr y_1^{(\!2\!)}  |x_1^{(\!3\!)}& y_{1}^{(\!3\!)}& y_{3}^{(\!3\!)}\cr  y_{1}^{(\!1\!)}& y_{1}^{(\!1\!)}
  \cr  y_{2}^{(\!2\!)}& y_{1}^{(\!2\!)}\cr  y_{1}^{(\!3\!)}| y_{2}^{(\!3\!)}& x_{1}^{(\!3\!)}& y_{2}^{(\!2\!)}& y_{1}^{(\!3\!)}
  \cr  x_{1}^{(\!2\!)}& y_{1}^{(\!2\!)}& x_{1}^{(\!3\!)}\cr  y_{2}^{(\!3\!)})\,\biggr).
\end{equation*}
Applying the  RSK superinsertion algorithm, we obtain
\begin{equation*}
  S=\biggl(\,\Tricolor(x_1^{(\!1\!)}&y_1^{(\!1\!)}\cr y_1^{(\!1\!)}\cr y_1^{(\!1\!)}|
  x_1^{(\!2\!)}&x_1^{(\!2\!)}&y_1^{(\!2\!)}\cr y_1^{(\!2\!)}&y_2^{(\!2\!)}\cr y_1^{(\!2\!)}\cr y_2^{(\!2\!)}|x_1^{(\!3\!)}&x_1^{(\!3\!)}&x_1^{(\!3\!)}&y_1^{(\!3\!)}&y_2^{(\!3\!)}
  \cr y_1^{(\!3\!)}&y_2^{(\!3\!)}& y_3^{(\!3\!)}\cr y_1^{(\!3\!)})\,\biggr),
\end{equation*}
\begin{equation*}
 T=\biggl(\,\Tricolor(1&3\cr 8\cr 9|
  2&4&10\cr 11&15\cr 17\cr 18|5&6&7&13&20
  \cr 12&14\cr 19\cr 16)\,\biggr).
\end{equation*}
Therefore 
\begin{eqnarray*}
   &&T_{\mathrm{SW}}=\{1,5,6,15,17,18\}, \\
  &&T_{\mathrm{NE}}=\{8,10,13,14\}, \text{ and } \\
  &&P_{T}^{\bmu}=\{2,3,4,7,8,10,12,13,19,20\}.
\end{eqnarray*}
So $\ell_0(T_{\mathrm{SW}})=\ell_1(T_{\mathrm{SW}})=3$, $\ell_0(T_{\mathrm{NE}})=1$,
$\ell_1(T_{\mathrm{NE}})=3$, and $\mathrm{wt}_{\bmu}(T)=q^{-2}Q_1^3Q_2^4Q_3^{13}=\mathrm{wt}_{\bmu}(\mathbf{i})$.
\end{example}

\begin{theorem}\label{Them:q=wt-Schur}For $\bmu\in\mpn$, we have
\begin{equation*}
  q_{\bmu}(\bx/\by;q,\bq)=\sum_{\bgl\in\mpn}\biggl(\sum_{\ft\in\mathrm{std}(\bgl)}
  \mathrm{wt}_{\bmu}(\ft)\biggr)s_{\bgl}(\bx/\by).
\end{equation*}
\end{theorem}
\begin{proof}Comparing Definitions~\ref{Def:mu-up-down} and \ref{Def:mu-weight} with Definitions~\ref{Def:mu-SW-NE} and \ref{Def:tableau-weight}, we see that the RSK superinsertion algorithm satisfies \begin{equation*}
  \text{if } (S,T)\xymatrix@C=0.65cm{\ar@{<->}[r]^{\substack{\mathrm{RSK}}}&}
  z_{\mathbf{i}}=\left(\begin{array}{ccc}1&\cdots& n\\
 z_{i_1}& \cdots&z_{i_n}\\ \end{array} \right) \text{ then } \mathrm{wt}_{\bmu}(\mathbf{i})=\mathrm{wt}_{\bmu}(T).
\end{equation*}
Thanks to Corollary~\ref{Cor:q-mu-wt}, we yield
\begin{eqnarray*}
  q_{\bmu}(\bx/\by;q,\mathbf{Q})&=& \sum_{\mathbf{i}=(i_1, \ldots, i_n)}\mathrm{wt}_{\bmu}(\mathbf{i})\mathbf{z}_{\mathbf{i}}\\
 &=&\sum_{\bgl\in\mpn}\sum_{(S,T)\in\mathrm{std}_{\bk|\bl}(\bgl)\times\mathrm{std}(\bgl)}
 \mathrm{wt}_{\bmu}(T)S(\bx/\by)\\
 &=&\sum_{\bgl\in\mpn}\sum_{T\in\mathrm{std}(\bgl)}
 \mathrm{wt}_{\bmu}(T)\sum_{S\in\mathrm{std}_{\bk|\bl}(\bgl)}S(\bx/\by)\\
 &=&\sum_{\bgl\in\mpn}\sum_{T\in\mathrm{std}(\bgl)}
 \mathrm{wt}_{\bmu}(T)s_{\bgl}(\bx/\by),
\end{eqnarray*}
where the last equality follows by applying Equ.~(\ref{Equ:Schur=tableau}).
\end{proof}

\begin{remark}\label{Remark:m=1}If $m=1$ then the super Frobenius formula Equ.~(\ref{Equ:char-Schur}) is  Mitsuhashi's super Frobenius formula for the characters of Iwahaori-Hecke algebra $H_n(q)$ (cf. \cite[Theorem~5.5]{M2010}):
For $\mu\vdash n$,
\begin{equation*}
  q_{\mu}(\bx/\by;q)=\sum_{\lambda\vdash n}\biggl(\sum_{\ft\in\mathrm{std}(\lambda)}
  \mathrm{wt}_{\mu}(\ft)\biggr)s_{\lambda}(\bx/\by).
\end{equation*}
Therefore Theorem~\ref{Them:q=wt-Schur} implies  for $\mu\vdash n$,
\begin{equation*}
  q_{\mu}(\bx/\by;q)=\sum_{\lambda\vdash n}\biggl(\sum_{\ft\in\mathrm{std}(\lambda)}
  \mathrm{wt}_{\mu}(\ft)\biggr)s_{\lambda}(\bx/\by),
\end{equation*}
which is a super-version of Ram's formula in \cite{Ram,Ram98}.
\end{remark}

Note that the super Schur functions $s_{\bgl}(\bx/\by)$ are linearly independent \cite[Lemma~6.4]{B-Regev}. Equ.~(\ref{Equ:char-Schur}) and Theorem~\ref{Them:q=wt-Schur} implies
\begin{corollary}\label{Cor:Character}
  For $\bgl,\bmu\in\mpn$, we have 
  \begin{equation*}
    \chi_q^{\bgl}(\bmu)=\sum_{T\in\mathrm{std}(\bgl)}\mathrm{wt}_{\bmu}(T).
  \end{equation*}
\end{corollary}

\begin{remark}\label{Remark:q=1}Putting $q=1$ and $Q_i=\varsigma^i$,   Theorem~\ref{Them:q=wt-Schur} enables us to obtain the following super-symmetric function identity
 \begin{equation*}
   P_{\bmu}(\bx/\by)=\biggl.\sum_{\bgl\in\mpn}\biggl(\sum_{\mathfrak{t}\in\mathrm{std}(\bgl)}\mathrm{wt}_{\bmu}
   (\mathfrak{t})\biggr|_{\substack{q=1\\Q_i=\varsigma^i}}\biggr)s_{\bgl}(\bx/\by)
 \end{equation*}
 and a character formula  \begin{equation*}
    \chi_1^{\bgl}(w_{\bmu})=\biggl.\sum_{\mathfrak{t}\in\mathrm{std}(\bgl)}
    \mathrm{wt}_{\bmu}(\mathfrak{t})\biggr|_{\substack{q=1\\Q_i=\varsigma^i}}
  \end{equation*}
for complex reflection group  $W_{m,n}$. Note that Cantrell et. al. \cite[Remark~3.5]{CHM} give a character formula for $W_{m,n}$.
 \end{remark}

\begin{remark}For $\bgl\in\mpn$, we denote by $d_{\bgl}$ the number of standard $\bgl$-tableaux, that is, $d_{\bgl}$ is the dimension of the irreducible representation of $H$ or $W_{m,n}$ corresponding to $\bgl$.  As a special case of the RSK superinsertion, we can restrict to sequences $\mathbf{z}=z_{i_1}^{(j_1)}\cdots z_{i_n}^{(j_n)}$ where $i_1, \ldots, i_n$ is a permutation of $n$ and $1\leq j_r\leq m$, that is, $k_i+\ell_i=n$ for $i=1, \ldots, m$. There are $n!m^n$ such sequences. Further, when we insert these special sequences, we get a pair $(S,T)$-standard tableaux (the $(\bk,\bl)$-semistandard tableau $S$ is standard because all the elements in $\mathbf{}$ are different). Thus the RSK superinsertion algorithm gives the following equality
 \begin{equation*}
   n!m^n=\sum_{\bgl\in H(\bk|\bl;m,n)}d_{\bgl}^2,
 \end{equation*}
 where $\bk$, $\bl$ satisfies $k_i+\ell_i=n$ for  $i=1, \ldots, m$. 
\end{remark}

\end{document}